\documentclass[11pt]{amsart}
\usepackage{amssymb,latexsym,amsmath,amscd,amsthm,graphicx}
\usepackage{etex,pictex}
\usepackage{setspace}
\usepackage[top = 1.0in, bottom = 1.0in, left =1.5in, right = 1.0in]{geometry}

\theoremstyle{plain}

\newtheorem{theorem}[subsection]{Theorem}

\newtheorem{lemma}[subsection]{Lemma}
\newtheorem{defi}[subsection]{Definition}
\newtheorem{prop}[subsection]{Proposition}

\newcommand{\Nat}{\mathbb N}

\newcommand{\Stack}{\mathcal{S}}

\begin{document}

\title{Nearly Continuous Even Kakutani Equivalence of Strongly Rank One Transformations}
%\title{A Title Will Be Here.... Eventually}%title goes here%
\author{Daniel J. Rudolph \\ Bethany D. Springer}%authors go here%

\date{}

\thanks{}%grant support etc. goes here%
 \begin{abstract}
In ergodic theory, two systems are Kakutani equivalent if there exists a conjugacy between induced transformations.  In Measured Topological Orbit and Kakutani Equivalence, del Junco, Rudolph, and Weiss  defined nearly continuous even Kakutani equivalence as an orbit equivalence which restricts to a conjugacy between induced maps on nearly clopen sets.  This paper defines nearly continuous even Kakutani equivalence between two nearly continuous dynamical systems as a nearly continuous conjugacy between induced maps on nearly clopen sets of the same size and shows that this definition is equivalent to the definition given by del Junco, Rudolph, and Weiss.  The paper shows that, with an added restriction, if two systems are nearly continuously (non-even) Kakutani equivalent, then one system is isomorphic to an induced transformation of the other, and uses this result to prove that a class of transformations containing Chacon's map and called strongly rank one belongs to the same nearly continuous even Kakutani equivalence class as irrational rotations.
%abstract of course goes here
 \end{abstract}

\maketitle
\pagestyle{headings}\markboth{}{}%page headings go here

\section{Introduction}
In ergodic theory, we ask properties to hold modulo sets of measure zero.  In the following theory, dubbed nearly continuous dynamics, we add a topological component by asking properties to hold modulo $F_\sigma$ sets of measure zero: transformations must be homeomorphisms in the relative topology of some $G_\delta$ set of full measure; and we ask sets to be nearly clopen instead of measurable, i.e. clopen in the relative topology when restricted to a $G_\delta$ subset of full measure.  A nearly continuous dynamical system consists of a Polish probability space and a transformation which is an ergodic measure preserving homeomorphism when restricted to a $G_\delta$ subset of full measure which is invariant under the transformation.  We consider two transformations to be equivalent if they agree upon a $G_\delta$ subset of full measure.  

In the 1940s, S. Kakutani introduced a new type of equivalence between ergodic measure preserving transformations of finite measure spaces relevant to studying the relationships between measurable cross sections of measurable flows.  Two ergodic measure preserving transformations $S$ and $T$ were called Kakutani equivalent if they had isomorphic induced maps, or equivalently, if $S$ and $T$ could be realized as flows over measurable cross sections of a third system.  

In 1984, A. del Junco and D. Rudolph extended the definition of even Kakutani equivalence to ergodic $\mathbb{Z}^n$ actions in \cite{AR}, first by taking Katok cross-sections of a flow, and second by determining the existence of an orbit preserving injection with an extra asymptotic linearity condition.  In one dimension, this translated to the classical theory of Kakutani equivalence with an addition of an orbit equivalence between the two systems.  When A. del Junco, D. Rudolph, and B. Weiss explored Kakutani equivalence in the context of what they named measured topological dynamical systems \cite{JRW}, this element of orbit equivalence persisted in the definition.  The authors defined even Kakutani equivalence to be an orbit equivalence on sets of full measure which gave a conjugacy when restricted to induced transformations of nearly clopen subsets of the same measure.  We give a slightly different definition for nearly continuous even Kakutani equivalence, asking only for the conjugacy between induced transformations on nearly clopen subsets of the same measure.  The first aim of the paper is to prove the following theorem which demonstrates that our definition of nearly continuous even Kakutani equivalence is synonymous with the definition given by del Junco, Rudolph and Weiss:

\begin{theorem}\label{NCEKE}
If two nearly continuous dynamical systems $(X, \tau, \mu, T)$ and $(Y, \tau, \nu,S)$ are nearly continuously even Kakutani equivalent, then the nearly continuous conjugacy given by $\varphi$ such that $\varphi \circ T_A = S_B \circ \varphi$, where $A \subset X$ and $B \subset Y$ are nearly clopen sets, extends to a nearly continuous orbit equivalence.
\end{theorem}

Rudolph and Roychowdhury \cite{RandR2} proved that all adding machines are nearly continuously even Kakutani equivalent, and Rudolph and Dykstra \cite{Dykstra} added irrational rotations to the class.  We prove the following in the case of non-even Kakutani equivalence, and apply it to show that a class of transformations which includes Chacon's map and is here-in named strongly rank one belongs to the equivalence class of irrational rotations.

\begin{theorem}\label{NCKE}

Let $\left(X, \tau, \mu, T \right)$ and $\left(Y,\tau,\nu, S \right)$ be two nearly continuous dynamical systems.  Let $A \subset X$ and $B \subset Y$ be nearly clopen subsets such that $\mu(A) > \nu(B)>0$ and $T_A$ n.c. conjugate to $S_B$.   If $\left(X, \tau, \mu, T \right)$ is nearly uniquely ergodic, then there exists a nearly clopen set $\bar B \subset Y$ such that $S_{\bar B}$ is nearly continuously conjugate to $T$.  

\end{theorem}

The structure of the paper is as follows:  section 2 gives basic definitions.  Section 3 provides proof of Theorem \ref{NCEKE}.  Section 4 defines nearly unique ergodicty and provides the proof of Theorem \ref{NCKE}.  Section 5 defines strongly rank one and shows they are nceke to irrational rotations.

\section{Definitions}  

We choose the state space $X$ to be an uncountable Polish space:  a separable topological space neither finite nor homeomorphic to $\Nat$ permitting at least one metric for which the space is complete.  The space is endowed with a non-atomic Borel probability measure $\mu$ of full support.  Recall that Polish spaces are $G_\delta$ subspaces of compact metric spaces.  The topology $\tau$ is the Polish topology, and whenever we write $(X,\tau,\mu)$ we mean precisely such a space, referred to as a \textit{Polish probability space.}

\begin{defi}
A $G_\delta$ subset of full measure $X_0 \subseteq X$ is called a \textbf{nearly full} subset.
\end{defi}

Because the support of $\mu$ is $X$, any nearly full subset is dense in $X$.   Property R of an object holds \textbf{nearly} on the space $X$ if there exists a nearly full subset $X_0 \subseteq X$ such that $R$ holds in the relative topology of $X_0$.  For instance:

\begin{defi} A set $C \subseteq X$ is said to be \textbf{nearly clopen in $X_0$} if there exists a nearly full subset $X_0 \subseteq X$ such that $C \cap X_0$ is clopen in the relative topology of $X_0$.  
\end{defi}

% Another characterization from (dJ,R,W lemma 2.7):  a set is nearly clopen iff it is nearly equal to a clopen set $A$ such that $\mu(\partial A) = 0$. 

\begin{defi} 
$A \subseteq X$ is \textbf{nearly equal} to $B \subseteq X $ in a Polish probability space $(X,\tau, \mu)$ if there exists a nearly full set $X_0$ such that $A \cap X_0 = B \cap X_0$.
\end{defi}

Two nearly equal sets have the same measure.  

\begin{prop}\label{AnclopenB}
If $(X,\tau,\mu)$ is a Polish probability space, $A \subseteq X$ is nearly clopen, and $B$ is nearly equal to $A$, then $B$ is nearly clopen.
\end{prop}

Near equality of nearly clopen sets is an equivalence relation.

Let $[A]$ denote the equivalence class of $A$, and let $\mathcal{A} = \mathcal{A}(x)$ denote the set of equivalence classes of nearly clopen sets of a Polish probability space $(X,\tau\,u)$.

\begin{prop}
$\mathcal{A}$ is a Boolean algebra.
\end{prop}

\begin{proof}
This follows from the facts that for nearly clopen sets $A$ and $B$ of $X$, $[A] \cap [B] = [A \cap B]$, $[A]^c = [A^c]$ and $[A] \cup [B] = [A \cup B]$.
\end{proof}
For assurance of the abundance of nearly clopen sets in Polish probability spaces, we refer the reader to \cite{JRW}.  

\begin{defi}
A function $f: X \rightarrow Y$ where $X$ is a Polish probability space is said to be \textbf{nearly continuous on $X_0$} if there exists a nearly full subset $X_0 \subset X$ such that $f|_{X_0}$ is continuous in the relative topology of $X_0$.  \end{defi}

We say that $f$ and $f'$ are nearly equal as functions on $X$ if there exists a nearly full set $X' \subseteq X$ such that $f|_{X'} = f|_X$.  

\begin{prop}
If $f$ is nearly continuous and $g$ is nearly equal to $f$, then $g$ is nearly continuous.
\end{prop}

The proof is straightforward, and near equality of functions gives an equivalence relation.

\begin{defi} Let $\phi:X \rightarrow Y$ be a transformation of Polish probability spaces $(X,\tau,\mu)$ and $(Y,\tau,\nu)$.  If there exists nearly full subsets $X_0 \subseteq X$ and $Y_0 \subseteq Y$ such that $\phi: X_0 \rightarrow Y_0$ is a homeomorphism, then $\phi$ is a \textbf{near homeomorphism} and the spaces are called \textbf{nearly homeomorphic}.
\end{defi}

Two near homeomorphisms $\phi$ and $\phi'$ are nearly equal if there exists $X_1 \subseteq X$ and $Y_1 \subseteq Y$, nearly full subsets, such that $\phi: X_1 \rightarrow Y_1$ is a homeomorphism and $\phi|_{X_1} = \phi'|_{X_1}$, $\phi^{-1}|_{Y_1} = \phi'^{-1}|_{Y_1}$.

We shall occasionally abbreviate nearly continuously as n.c..

\begin{defi}
Let $(X,\tau,\mu)$ be a Polish probability space.  A nearly continuous dynamical system $(X, \tau, \mu, T)$ is given by a subset $X_0 \subseteq X$ and a map $T: X_0 \rightarrow X_0$ such that:
\begin{enumerate}
\item $X_0$ is a nearly full, $T$-invariant subset of $X$ and;

\item $T$ is an ergodic homeomorphism in the relative topology of $X_0$ preserving the measure, $\mu$, restricted to $X_0$.

\end{enumerate}
\end{defi}

\begin{defi}
Let $(X,\tau,\mu,T)$ be a nearly continuous dynamical system.  A nearly full subset is called a \textbf{carrier} if it is invariant under $T$.
\end{defi}

\begin{prop}\label{carrier}
The intersection of two carriers is a carrier.
\end{prop}

\begin{proof}
Let $(X,\tau,\mu,T)$ be a nearly continuous dynamical system, and let $X_1$ and $X_2$ be two carriers for $T$.  Then $T^{-1}(X_1 \cap X_2) = T^{-1}(X_1) \cap T^{-1}(X_2) = X_1 \cap X_2.$
\end{proof}

\begin{defi}
For a Polish probability space $(X,\tau,\mu)$, let $\mathcal{G} = \mathcal{G}(X) = \{ (X, \tau, \mu, T) \}$, the set of all nearly continuous dynamical systems on $X$.
\end{defi}

Two transformations $T \in \mathcal{G}$ and $S \in \mathcal{G}$ are equivalent if they have a common carrier $X'$ such that $T|_{X'} = S|_{X'}$.  This is an equivalence relation.

Nearly clopen sets play an important role in the study of induced transformations.  Let $(X,\tau,\mu,T)$ be a nearly continuous dynamical system, and $A \subseteq X$ be nearly clopen.  $T$ is only defined on a carrier $X_0 \subseteq X$, and $A \cap X'$ is clopen in the relative topology of a nearly full set $X' \subseteq X$.  It is possible that $T$ is not defined on $A$ or on $A \cap X'$.  Let $X_1 = \cap_{i = \infty}^\infty T^{-i}(X_0 \cap X')$, giving a carrier of $T$ such that $A_1 = A \cap X_1 \in [A]$ is clopen in $X_1$and $T$ is defined on $A_1$.  Then, define the return time function $r_{A_1}: A_1 \rightarrow \mathbb{N}$ by $r_{A_1}(x) = \inf \{ r>0 : (T|_{X_1})^r (x) \in A_1 \}$.

\begin{lemma}\label{returntime} The return time function to a nearly clopen set is nearly continuous.  
\end{lemma}

\begin{proof}
Let $A_1$ and $X_1$ be as above.  Decompose $A_1$ into the sets with the same first return times:  
\[ \begin{array}{rcl} 
B_1 & = & (T|_{X_1})^{-1}(A_1) \cap A_1 \\ 
B_2 & = & \left( (T|_{X_1})^{-2}(A_1) \cap A_1 \right) \setminus B_1, \\ 
{} &  \ldots & {} \\  
B_n & = & \left( (T|_{X_1})^{-n}(A_1) \cap A_1 \right) \setminus \bigcup_{i=1}^{n-1}B_i.
\end{array}\]
  Each $B_i$ is clopen in $X_1$, and $r_{A_1}(x)= i$ for $x \in B_i$ so that $r_{A_1}$ is constant on each clopen set of the decomposition, and therefore nearly continuous.

\end{proof}

For a nearly clopen set $A$, the induced map, defined on a subset $A_1$ of $A$, is given by $T_{A_1} = T^{r_{A_1}}$.  Note that $A_1$ is a $G_\delta$ subset of $X$ so that $(A_1, \tau|_{A_1}, \mu|_{A_1})$ is a Polish probability space.

\begin{lemma}   $T_{A_1}$ is a homeomorphism in the relative topology of $A_1$
\end{lemma}

\begin{proof}
Use the decomposition in the proof of Lemma \ref{returntime} of $A_1$ into $B_i$.  For $x \in B_i$, $r_{A_1}(x) = i$ so that $T_{A_1} = T^i$ on $B_i$ , and the induced map is a piecewise function of homeomorphisms on a clopen decomposition of $X_1$.  
\end{proof}

Note that the set $X_1$ is not unique.  We refer to $T_A$ as a \textbf{nearly continuous induced map} because there exists a carrier $X_1$ such that $A \cap X_1 = A_1$ is clopen in $X_1$ and $T_{A_1}$ is a homeomorphism on $A_1$.  
\begin{defi}
For $(X,\tau,\mu,T)$, a n.c. dynamical system, and the nearly clopen set $A \subset X$, the \textbf{skyscraper} over $A$ is the ordered list of sets \[ \begin{array} {cccccc} A,& \ T(A) \backslash A, & \, T^2 (A) \backslash (A \cup T(A)),& \ldots , & T^k(A) \backslash \bigcup_{i=0}^{k-1} T^i(A), & \ldots \end{array} \] referred to as levels of the skyscraper, and $A$ is referred to as the base. 
\end{defi}
The levels are commonly visualized as intervals, with $A$ on the bottom, $T(A) \backslash A$ above, etc.  The transformation $T$ moves points in a level to the level above.  If the image of a point is not in the next level, then the image appears in the base.  For an $x \in A$, $T$ moves the point $x$ up the skyscraper to height $r_A(x)$, then to the point $T_A(x)$.  

\begin{defi}
Let $(X,\tau,\mu,T)$ be a n.c. dynamical system and $B$ be a nearly clopen set.  A \textbf{nearly clopen tower} of height $n$ over $B$ is the set of $n - 1$ images of $B$ under $T$: \[ B, T(B), T^2(B), \ldots, T^{n-1}(B) \] such that $\mu(T^i(B) \cap T^j(B)) = 0$ for $0 \leq i < j \leq n-1$.  The set $B$ is called the \textbf{base} of the tower.
\end{defi}

\begin{defi}
A \textbf{column} $\mathcal C$ of a skyscraper over a nearly clopen set $A$ is an ordered list of sets \[ B, T(B), T^2(B), \ldots, T^{r-1}(B) \] where $r_A(x) = r$ for all $x \in B$.  $B \subset A$ is the base of the column.
\end{defi}

As the return time function is nearly continuous, any skyscraper may be decomposed into columns with constant return times for each base such that the levels are nearly clopen.
We now turn our attention to the various notions of equivalence between two nearly continuous dynamical systems.

\subsection{Nearly Continuous Equivalences}
\begin{defi}
Two systems $(X,\tau,\mu,T)$ and $(Y,\tau,\nu,S)$ are said to be \textbf{nearly continuously conjugate} if there exists carriers $X_0 \subset X$ and $Y_0 \subset Y$ and a homeomorphism $\phi: X_0 \rightarrow Y_0$ such that $\phi \circ T|_{X_0} = S|_{Y_0} \circ \phi$.
\end{defi}

\vspace{.2in}
\begin{defi}
For a transformation $T: X \rightarrow X$, the $\textit{orbit}$ of a point $x \in X$ is \[ Orb_T(x) = \{ T^i(x) : i \in \mathbb{Z} \}. \]
\end{defi}

\begin{defi}
Fix a $T \in \mathcal{G}(X)$ and select $S \in \mathcal{G}(X)$ such that $Orb_S(x) \subseteq Orb_T(x)$.  The \textbf{cocycle} of $S$ is the corresponding $\mathbb{Z}$-valued map $C_S: X \rightarrow \mathbb{Z}$ such that $S(x) = T^{C_S(x)}(x)$.
\end{defi}

\begin{defi}
For a n.c. dynamical system $(X,\tau,\mu, T)$, the \textbf{nearly continuous full group} $[[ T ]] \leqslant \mathcal{G}(X)$ is the set of ergodic measure preserving near homeomorphisms $T'$ such that there exists a common carrier $X_0 \subseteq X$ for $T$ and $T'$ with $Orb_{T'}(x) =  Orb_T(x)$ for $x \in X_0$ where the cocycle $C_{T'}$ is nearly continuous.
\end{defi}

\begin{defi} Two nearly continuous dynamical systems $(X,\tau,\mu,T)$ and $(Y,\tau, \nu, S)$ are \textbf{nearly continuously orbit equivalent} if there exists $T' \in [[T]]$ and $S' \in [[S]]$ such that $T'$ is nearly continuously conjugate to $S'$.
\end{defi}

Thus, if the two systems are n.c. orbit equivalent, there exist carriers $X_0 \subseteq X$ and $Y_0 \subseteq Y$, a homeomorphism $\phi: X_0 \rightarrow Y_0$, and continuous maps $p: X_0 \rightarrow \mathbb{Z}$ and $q: Y_0 \rightarrow \mathbb{Z}$ such that $\phi \circ T^p = S^q \circ \phi$ and $\phi (Orb_T(x)) = Orb_S(\phi(x))$ for $x \in X_0$.

% Suppose $(X,\tau,\mu,T)$ is n.c. orbit equivalent to $(Y,\tau,\nu,S)$.  Then, there exists $T' \in [[T]]$ and $S' \in [[S]]$ such that $T'$ is n.c. conjugate to $S'$.  Then $[T'] \subset \left[ \left( [[T]] \right) \right]$ is n.c. conjugate to $[S'] \subset \left[ \left( [[T]] \right) \right]$ $\Rightarrow$ $[T'] \subset [[ \left( [T] \right) ]]$ is n.c. conjugate to $[S'] \subset [[ \left( [S] \right) ]]$.

% Also, if $\hat T \in [T]$, then $[\hat T] = [T]$
% \[ \Rightarrow [T'] \subseteq [[ \left( [\hat T] \right) ]] is n.c. conjugate to [S'] \subseteq [[ \left( [\hat S] \right) ]]  \]
% \[ \Rightarrow [T'] \subseteq [ \left( [[\hat T]] \right) ] is n.c. conjugate to [S'] \subseteq [ \left( [[\hat S]] \right) ],  \]
% $\Rightarrow$ there exists $\hat T' \in [T'] \subset [[\hat T]]$ and $\hat S' \in [S'] \subset [[\hat S]]$ such that $\hat T'$ and $\hat S'$ are n.c. conjugate.
%\end{proof}

In the nearly continuous category, A. del Junco and A. \c Sahin proved in \cite{dJS} that Dye's theorem holds:

\begin{theorem}(J.\c S.) Suppose $(X, \tau, \mu, T)$ and $(Y,\tau,\nu, S)$ are nearly continuous dynamical systems.  Then, the systems are nearly continuously orbit equivalent.
\end{theorem}

\begin{defi} Two nearly continuous dynamical systems $(X, \tau,\mu, T)$ and $(Y, \tau, \nu, S)$ are \textbf{nearly continuously Kakutani equivalent} if there exist nearly clopen sets $A \subseteq X$ and $B \subseteq Y$ with $\mu(A) >0$ and $ \nu(B) > 0$ such that $T_A$ and $S_B$ are nearly continuously conjugate as induced maps.

If $\mu(A) = \nu(B)$, then the systems are \textbf{ nearly continuously even Kakutani equivalent }.
\end{defi}

This concludes the introduction of spaces, maps, sets, and notions of equivalence employed throughout the paper.

\section{Proof of Theorem \ref{NCEKE}}

As $(X, \tau,\mu, T)$ and $(Y, \tau, \nu, S)$ are n.c. even Kakutani equivalent, there exists $X_0 \subseteq X$ and $Y_0 \subseteq Y$, carriers of $T$ and $S$, respectively, such that $A_0 = A \cap X_0$ and $B_0 = B \cap Y_0$ are clopen in $X_0$ and $Y_0$, respectively, such that $\phi: A_0 \rightarrow B_0$ is a homeomorphism and $\phi \circ T_{A_0}= S_{B_0} \circ \phi$.  Our goal is to extend $\phi$ to a carrier of $T$ in such a way so as to establish an orbit equivalence, $\hat \phi$, while (nearly) preserving the conjugacy between the induced maps.  We do so by defining a point map constructed piecewise on a nearly clopen decomposition of $X_0$ (based on return times to $A_0$ and $B_0$), where $\hat \phi$ is a composition of the homeomorphisms $T$, $T_{A_0}$, $\phi$, and $S$ on each piece.  To help make sense of the definition of the extension, we begin with a description of the machinery used to create the map.

\subsection{Describing Piles and Pits}

Construct the nearly clopen skyscraper of $T$ over the set $A_0$.  The height of the tower over a point $x \in A_0$ is 
\[r_{A_0}(x) := \min\{n > 0 : T^n(x) \in A_0\}.\]  For any $x \in X_0$, define \[h(x) = \min \{ h >0 : T^{-h}(x) \in A_0 \}.\]  Let $\tilde x = T^{-h(x)}(x)$.  The fiber of the tower containing the point $x$ is the list of points \[ \left\{ \begin{array}{c} T^{r_{A_0}(\tilde x)}(\tilde x) \\ \vdots \\ x \\ \vdots \\ T(\tilde x) \\ \tilde x \end{array} \right\} .\]  We refer to these fibers as piles.  To the right of the pile containing $x$, list the fiber above $T_{A_0} \circ T_A(\tilde x)$.  Continue listing the piles so that to the right of each base point $\tilde x \in A_0$, one sees the fiber above $T_{A_0}( \tilde x)$, and, to the left of $\tilde x \in A_0$ one sees the fiber above $T_{A_0}^{-1}( \tilde x)$.  This ordered list of fibers gives what we call the upper frame for the point $x$, simply slices of the skyscraper as we follow the orbit of $x$.

For each point in $A_0$, there exists a point in $B_0$ via the conjugacy $\phi$.  For each of these points, we extract a fiber of the Kakutani skyscraper built via $S$ over the set $B_0$.  We denote by $r_{B_0}(\tilde y)$ the height of the tower over a point $\tilde y \in B_0$, \[r_{B_0}(\tilde y) := \min\{n > 0 : S^n(\tilde y) \in B_0\}. \]  Simply arrange the fibers from the skyscraper over $B_0$ in the same manner as for the fibers above points in $A_0$, with the fiber for $S_{B_0}(\tilde y)$ listed immediately to the right of the fiber for $\tilde y \in B_0$, and the fiber for $S_{B_0}^{-1}(\tilde y)$ listed immediately to the left of the fiber for $\tilde y \in B_0$.  Now, ``flip" the fibers upside-down over the points in $B_0$ so the point at the top of each list is in $B_0$.  We refer to these fibers as pits.

For the last step in this visualization, line up the base of the piles with the top of the pits via the conjugacy.  One traverses the diagram in the following manner:  applying the induced map $T_{A_0}$ to the bottom of the upper frame shifts the points to the right along the base, and applying the induced map $S_{B_0}$ to the top of the lower frame shifts the points to the right across the top.  Apply $T$ to move up the piles and apply $S$ to move down the pits.  Cross between the upper frame and the lower frame by applying $\phi$ or $\phi^{-1}$.

\begin{figure}[!h]
	\begin{center}
		\includegraphics[height= 5cm]{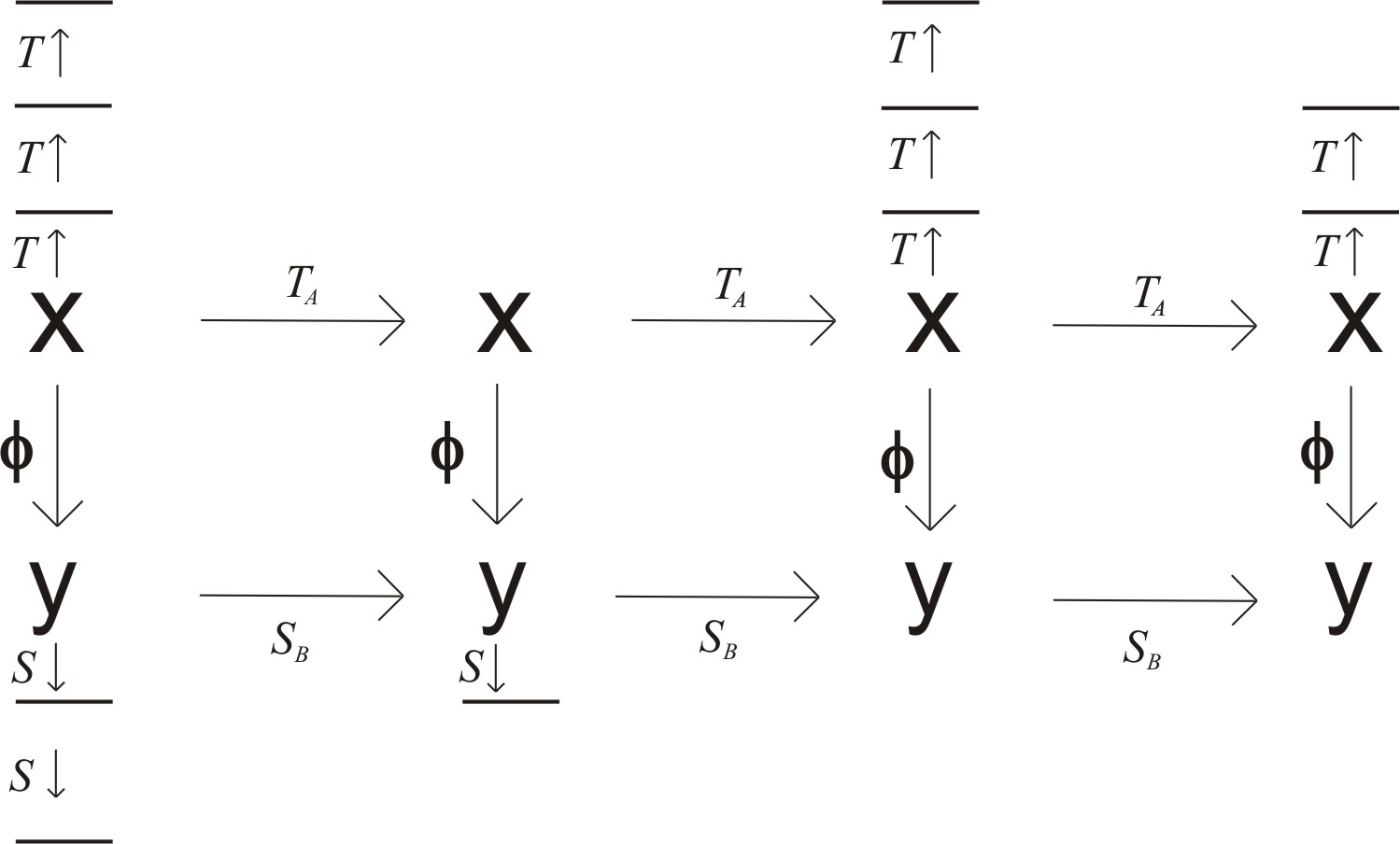}
	\end{center}
\end{figure}

We create an orbit equivalence by developing a method for assigning points from the piles to points in the pits.  The technique we describe is much like that of a machine in a factory.  This machine deposits items from the piles into the pits directly below until either the pile is empty (to all points in the pile, a map has been assigned) or the pit is full.  Then, it merely shifts the upper frame one to the right and deposits as much of the remainder of the piles into the remainder of the pits, shifts again and deposits again, etc.  
Here, we merely shift the upper frame across the lower by utilizing the induced map on the base of the frames.

A point $x \in A_0^C$ maps to a point $y \in B_0^C$ if, after a shift, the pit below $x$ has enough as of yet unused space to receive the point.  A point in a pile cannot be mapped to a pit if after a shift, the pit below it is full.

Define 
\[ \theta_k^n (x) = \sum_{i=k}^n r_{A_0}(T_{A_0}^i \circ T^{-h(x)}(x)) \] the sum of the heights of the $k^{th}$ pile through the $n^{th}$ pile to the right of the pile containing $x$, and

\[ \psi_k^n (x) = \sum_{i = k}^n r_{B_0}( S_{B_0}^i \circ \phi \circ T^{-h(x)}(x)), \] the sum of the depths of the $k^{th}$ pit through the $n^{th}$ pit to the right of the pit below the pile containing $x$.

 Let $n(x) = 0$ if $h(x) \leq r_{B_0} (\phi \circ T^{-h(x)}(x))$.  Otherwise, let \[ n(x) = \min \{n > 0: h(x) + \theta_1^n (x) \leq \psi_0^n (x) \}.\]  In essence, this is the smallest number of shifts required for there to be space available in a pit for the point $x$ in a pile.

Let $d(x)$ denote the depth in the pit to which $x$ is mapped.  We find $d(x)$ by calculating how much of the receiving pit has already been filled by other points, which comes from looking at the difference between stack heights and pit depths:
\[d(x) := h(x) + \theta_1^{n(x)} (x) - \psi_0^{n(x)-1} (x).\]

Define $\hat \phi$ as:

\[  \hat \phi (x) = \left\{ \begin{array}{lcl}   \phi(x) & \mbox{for} & x \in A_0 \\  S^{d(x)} \circ \phi \circ T_A^{n(x)} \circ T^{-h(x)}(x) & \mbox{for } & x \mbox{ otherwise} \end{array} \right. \]
					
In a similar manner, we define $\hat \phi ^{-1}$.
\[D(y) = \min \{ D >0: S^{-D}(y) \in B_0 \} \] is the depth below the stack at which the point $y$ sits.
Let \[ \gamma_k^n(y) = \sum_{i = k}^n r_{B_0}(S^{-i}_{B_0} \circ S^{-D(y)}(y) ), \] the sum of the depths of the $k^{th}$ through the $n^{th}$ pits to the left of the pit holding $y$, and let \[ \Psi_k^n(y) = \sum_{i = k}^n r_{A_0}(T_{A_0}^{-i} \circ \phi^{-1} \circ S^{-D(y)}(y)), \] the sum of the heights of the $k^{th}$ through the $n^{th}$ piles to the left of the pile above the pit containing $y$.  The inverse map pairs a point in a pit with a point in a pile by sliding the lower frame backwards along the upper frame, in reverse factory-machine manner.  The machine deposits as much as possible, shifts, deposits again, etc.  The transformation is defined at the point $y$ when, after a certain number of shifts, the pile about has available space.  Let $m(y)$ be the number of shifts needed for a point $y \in B_0^c$.  If $D(y) \leq r_{A_0} (\phi^{-1} \circ S^{-D(y)}(y)$, then $m(y) = 0$.  Otherwise,
\[ m(y) = \min \{ m > 0: D(y) + \gamma_1^m(y) \leq \Psi_0^m(y). \} \]  The height in the receiving pile for $y$ is then given by \[ H(y) = D(y) + \gamma_1^m(y) - \Psi_0^m(y).\]
\[  \hat \phi^{-1} (y) = \left\{ \begin{array}{lcl}   \phi^{-1}(y) & \mbox{for} & y \in B_0 \\  T^{H(y)} \circ \phi^{-1} \circ S_{B_0}^{-m(y)} \circ S^{-D(y)}(y) & \mbox{for } & y \mbox{ otherwise} \end{array} \right. \]

\begin{lemma} $\hat \phi ^{-1}$ is the inverse of $\hat \phi$.
\end{lemma}
We give a heuristic argument.  Let us again imagine the construction as a machine which deposits items from, say, storage bottles to spots in packing bins.  For $\hat \phi$, we assign an item at a specific storage location in a storage bottle to the first available empty space in a packing bin as we shift the storage bottles forward.  For $\hat \phi^{-1}$, we assign an empty space in a packing bin to the first available item in a storage bottle shifting the storage bottles backwards.  The diagram of the $\hat \phi ^{-1}$ construction is just the diagram of the $\hat \phi$ construction rotated 180$^{\circ}$.  Naturally, $D(\hat \phi(x)) = d(x)$, the same number of shifts are required.  Then, it is easily follows that $H(\hat \phi(x)) = h(x)$.  

$\hat \phi$ and $\hat \phi ^{-1}$ need to be well-defined on nearly full sets.  In order for $\hat \phi$ to be well-defined at a point, the values $h(x)$, $n(x)$, and $d(x)$ exist and must be finite.  We may assume that $h(x)$ (and $D(y)$) are finite.  Otherwise, we remove the entire orbit of $x$ from $X_0$ (or $y$ from $Y_0$).  To see that $n(x)$ (and $m(y)$) are finite on a subset of full  measure, we begin with the following lemma:

\begin{lemma}  Suppose $(X,\tau,\mu,T)$ is a nearly continuous dynamical system and $f(x)$ is a nearly continuous function such that $\int f(x) = 0$.  Then, for almost every $x \in X$, there exists $n >0$ such that $f(x) + f(T(x)) + f(T^2(x))+ \cdots + f(T^{n-1}(x)) \leq 0$.
\end{lemma}

\begin{proof}

Suppose not, that there exists $E \subseteq X$ with $\mu(E) >0$, such that for all $x \in E$ and all $n > 0$, $\sum_{i=0}^{n-1} T^i(f(x)) >0$.  

Let $x \in E$.  For a fixed value, $n$, calculate how many times the orbit of $x$ has returned to $E$.  Let $m(n) = \sum_{i=0}^{n-1} \chi_E (T^i(x))$.

Define the function $f_E(x) = \sum_{i=0}^{r_E(x) -1} f(T^i(x))$.  By assumption, $f_E$ is a positive function over $E$ and $T_E$ is ergodic $\Rightarrow$
\[ \lim_ {m \rightarrow \infty} \frac{1}{m} \sum_{i =0}^{m-1} f_E(T_E^i(x)) \rightarrow \int_E f_E d\mu _E > 0. \]

Also, \[ \sum_{i=0}^{m(n)-1} f_E(T_E^i(x)) = \sum_{i=0}^{n-1}f(T^i(x)) \] whenever $\chi_E(T^{n}(x)) = 1$.

For a subsequence $n_k$ of $n$ such that $\chi_E(T^{n_k}(x)) = 1$,
\[
\begin{array}{rcl} \lim_ {n \to \infty} \frac{1}{n_k} \sum_{i=0}^{n_k-1}f(T^i(x)) & = & \lim_{n_k \to \infty} \frac{1}{n_k} \sum_{i=0}^{m(n_k)-1}f_E(T_E^i(x)) \\ 
{} & = & \lim _{n_k \to \infty} \frac{1}{m(n_k)} \sum_{i=0}^{n_k-1} \chi_E (T^i(x)) \cdot \frac{1}{n_k} \sum_{i=0}^{m(m_k)-1}f_E(T_E^i(x)) \\ 
{} & = & \lim_ {n \to \infty} \frac{1}{n_k}\sum_{i=0}^{n_k-1} \chi_E (T^i(x)) \cdot \lim _{m \to \infty} \frac{1}{m} \sum_{i=0}^{m-1}f_E(T_E^i(x)) 
\\ {} & = & \mu(E) \int_E f_E d\mu _E \\ {} &>&0 \end{array} 
\]

As $\frac{1}{n} \sum_{i=0}^{n-1} f(T^i(x))$ converges, any subsequence converges to the same value, implying that $\int f >0$, a contradiction.
\end{proof}

\begin{lemma}  The value $n(x)$ is finite on a set of full measure.

\end{lemma}

\begin{proof}  For $\hat \phi(x)$ let $f(x) = r_{A_0}(x) - r_{B_0}(\phi(x))$ as in the previous lemma.  Let $X = A_0$, identified with $B_0$, and $T = T_{A_0} \cong S_{B_0}$.  As $\mu (A) = \nu (B)$, $\int f(x) =0$.  Thus $n(x)$ is finite for almost every $x$.

\end{proof}
The proof that $\hat \phi^{-1}$ is defined on a set of full measure is essentially identical.  
We now define the sets $\hat A$ and $\hat B$ upon which $\hat \phi$ and $\hat \phi^{-1}$ are well-defined.  First, let $A_{n,h,d}$ be the set of points for which $\hat \phi = S^d \circ \phi \circ T_A^n \circ T^{-h}$.  $A_{n,h,d}$ may tediously be written as a finite union of finite intersections of images of $A_0$ and $B_0$ via the homeomorphisms $S, \phi,$ and $T$.  Thus $A_{n,h,d}$ is clopen in $X_0$. Let $B_{m,D,H} = \{ y \in Y_0: \hat \phi^{-1}(y) = T^H \circ \phi^{-1} \circ S_{B_0}^{-m} \circ S^{-D} (y) \}$ so $B_{m,D,H} = \hat \phi (A_{m,H,D})$.  Define \[ \hat A = \bigcup_{n=0}^{\infty} \bigcup_{h=0}^\infty \bigcup_{d = 0}^{\infty} A_{n,h,d} \] and \[ \hat B = \bigcup_{m=0}^{\infty} \bigcup_{D=0}^\infty \bigcup_{H = 0}^{\infty} B_{m,D,H}. \]  
\begin{lemma}
$\hat A$ and $\hat B$ are nearly full sets.
\end{lemma}

\begin{proof}
$\hat A$ and $\hat B$ are both open in $X_0$ as they are the countable union of clopen sets in $X_0$, and of full measure by the previous lemma.
\end{proof}

\begin{lemma}

$\hat \phi$ is continuous on $\hat A$ and $\hat \phi^{-1}$ is continuous on $\hat B$.

\end{lemma}

\begin{proof}  Decompose $\hat A$ into the countable union of clopen sets $A_{h,n,d}$.  On each of these clopen sets, $\hat \phi$ is the composition $S^d \circ \phi \circ T_{A_0}^n \circ T^{-h}$ of functions continuous in the relative topology of $X_0$.  Decomposing $\hat B$ in a similar manner shows that $\hat \phi ^{-1}$ is also continuous in the relative topology of $\hat B$.
\end{proof}

\begin{lemma}

$\hat \phi$ is measure preserving on $\hat A$ and $\hat \phi ^{-1}$ is measure preserving on $\hat B$.

\end{lemma}

\begin{proof} 

Using the same decomposition, $\hat \phi$ is expressed on each nearly clopen set as the composition of measure preserving functions.  Likewise for $\hat \phi^{-1}$.
 
\end{proof}

We define the final carriers.  Let \[ X^* = \bigcap_{i = -\infty}^{\infty} T^{-i} (\hat A) \cap \hat \phi ^{-1} \left( \bigcap_{i = -\infty}^{\infty} S^{-i}(\hat B) \right)\] and \[ Y^* = \hat \phi (X^*). \]

\begin{lemma}

The sets $X^*$ and $Y^*$ on which $\hat \phi$  and $\hat \phi ^{-1}$ are defined are $G_\delta$ sets of full measure.

\end{lemma}

\begin{proof}
Note that $\hat A$ and $\hat B$ are both open sets of measure $1$. 

\end{proof}

In conclusion of Theorem \ref{NCEKE}, we have the existence of carriers $X^* \subseteq X$, $Y^* \subseteq Y$, $T^* = T_{A_0}^{n} \circ T^{-h} \in [[T]]$, $S^* = S^{d} \in [[S]]$ and a homeomorphism $\hat \phi: X^* \rightarrow Y^*$ such that $\hat \phi \circ T^* = S^* \circ \hat \phi$ for $x \in X^*$ and $T_A$ is n.c. conjugate to $S_B$.

\section{Nearly Unique Ergodicity and Non-even n.c. Kakutani Equivalence}

In this section, we begin by defining nearly unique ergodicity, a concept introduced by Denker and Keane in \cite{DK} as strict ergodicity and later re-named nearly unique ergodicity by del Junco, Rudolph, and Weiss in \cite{JRW} to show that n.c. Kakutani equivalence differs from (measured) Kakutani equivalence.  While we include the necessary facts, for more detail and discussion, see \cite{JRW} and \cite{DK}.

\begin{defi}
A sequence of functions $\{f_i\}_{i \in \Nat}$ on a Polish probability space $(X,\tau,\mu)$ is said to converge \textbf{nearly uniformly} if there exists a nearly full subset $X_0 \subseteq X$ on which the averages converge uniformly.
\end{defi}

\begin{defi}
Suppose $\left( X, \tau, \mu, T \right)$ is a nearly continuous dynamical system and suppose $f \in L_1(\mu)$.  If the ergodic averages of $f$ \[ A_n(f) = \frac{1}{n} \sum_{i = 0}^{n-1} f(T^i) \] converge nearly uniformly to $\int f d\mu$ for all nearly bounded and nearly continuous functions $f$, we say $T$ is \textbf{nearly uniquely ergodic}.

\end{defi}
Note that this definition differs slightly from that given in \cite{JRW} by asking for near uniform convergence of all nearly continuous functions.

\begin{theorem}
If $(X,\mu,T)$ is a uniquely ergodic system on a compact metric space, then it is nearly uniquely ergodic.  
\end{theorem}

\begin{proof}
Let $f$ be a continuous function on $X$.  The ergodic averages of $f$ converge uniformly, hence nearly uniformly.  Note that if the ergodic averages of $f$ converge nearly uniformly and $f'$ is nearly equal to $f$, then the ergodic averages of $f'$ converge nearly uniformly.  It now follows from the proof of Theorem 5.2 in \cite{JRW} that for any nearly clopen set $A \subset X$, the ergodic averages of $1_A$ converge nearly uniformly to $\mu(A)$.  By Lemma 5.3 of \cite{JRW}, $T$ is nearly uniquely ergodic.

\end{proof}

If a system $(X,\tau,\mu,T)$ is nearly continuously conjugate to a uniquely ergodic system, then it is nearly uniquely ergodic.  The converse is true as well.

For an example that nearly continuous Kakutani equivalence is strictly stronger than measured Kakutani equivalence (measurable conjugacy between induced systems on measurable sets), we refer the reader to section 6 of \cite{JRW}.

\subsection{The Proof of Theorem \ref{NCKE}}

First of all, let $X_0$ and $Y_0$ be carriers for $T$ and $S$, respectively, such that $A_0 = A \cap X_0$ and $B_0 = B \cap Y_0$ are clopen in their relative topologies, and let $\phi: A_0 \rightarrow B_0$ be the homeomorphism giving the conjugacy between $T_{A_0}$ and $S_{B_0}$.  The idea for the proof is to select subsets $A' \subseteq A_0$ and $B' \subseteq B_0$ such that the return times for points in $A'$ are smaller than the return times for the images of these points in $B'$.  (This equates to choosing sets for the bases of the frames so that the piles are all shorter than the pits with which they are paired.)

By Theorem 5.4 of \cite{JRW}, as $(X,\tau,\mu,T)$ is nearly uniquely ergodic and $A_0$ is a nearly clopen set with of positive measure, $T_{A_0}$ is nearly uniquely ergodic.  We use the nearly unique ergodicity of $T_{A_0}$ to find a set with ``nice" return times.  For $x \in A_0$, let \[r_{A_0}(x) = \min \left\{ r >0: T^r(x) \in A_0 \right\}. \]  Given an $\epsilon > 0$ there exists an $N_1 = N_1(\epsilon)$ such that \[| \frac{1}{n} \sum_{i = 1}^n r_{A_0} \left( T_{A_0}^i \right) - \int_{A_0} r_{A_0} d\mu_{A_0}| < \epsilon \] for all $n > N_1$ and almost all $x \in {A_0}$.

As $T_{A_0}$ is n.c. conjugate to $S_{B_0}$, by Corollary 8 from \cite{DK}, $S_{B_0}$ is nearly uniquely ergodic on $B_0$.  For $y \in B_0$, let \[r_{B_0}(y) = \min \left\{ r >0: S^r(y) \in B_0 \right\}.\]  Given an $\epsilon > 0$ there exists an $N_2 = N_2(\epsilon)$ such that \[ | \frac{1}{n} \sum_{i = 1}^n r_{B_0} \left( S_{B_0}^i \right) - \int_{B_0} r_{B_0} d\nu_{B_0}| < \epsilon \] for all $n > N_2$ and almost all $y \in B_0$.

Let $N = \max \{ N_1, N_2 \}$.  Select $\tilde A \subset A_0$ nearly clopen with $\mu_{A_0}(\tilde A) < \frac{1}{2N}$.  There exists $A^\prime \subset \tilde A$ s.t. the return time under $T_{A_0}$ to $A^ \prime$ is greater than or equal to $2N$ for every point, and $A^ \prime$ is nearly clopen.  

Thus, \[ r_{A^\prime}(x) \geq \sum_{i = 1}^{2N} r_{A_0}(T_{A_0}^i(x)) .\]

Let $B^\prime = \phi(A^\prime )$.  Note that for each $y \in B^\prime$, there exists an $x \in A^ \prime$ such that $\phi (x) = y \Rightarrow$  
\[ r_{B^\prime}(y) \geq \sum_{i = 1}^{2N} r_{B_0} (S_{B_0}^i (y)) .\] 

Let \[h(x) = min \{h >0: T^{-h}(x) \in A' \}. \]  Define
\[ \hat \phi := \left\{ \begin{array}{lcl} \phi(x) & \mbox{ if } & x \in A' \\
											S^{h(x)} \circ \phi \circ T^{-h(x)}(x) & \mbox{ if } & x \in (A')^C \cap X_0\end{array} \right. \]
											
Let \[ \bar B = \hat \phi (X). \]

\begin{lemma}
$\hat \phi$ is defined on a nearly full subset of $X$.
\end{lemma}

\begin{proof} Think in terms of piles and pits as in the proof of Theorem \ref{NCEKE}.  Here, we chose $A^\prime$ to be the base of the piles and $B^\prime$ to be the top of the pits.  So long as all of the pits are deeper than the piles with which they are paired via $\phi$, we may move almost all of $X$ into $Y$.  By choice of $A'$ and $B'$, for points in these sets, the return times to the set are at least $\sum_{i = 1}^{2N} r_{A_0}(T_{A_0}^i(x))$ and $\sum_{i = 1}^{2N} r_{B_0}(S_{B_0}^i(x))$ respectively.  By nearly unique ergodicity of $T_{A_0}$, as $2N > N$, we have that:
\[ \begin{array}{rrcl} {}&| \frac{1}{2N} \sum_{i = 1}^{2N} r_{A_0}(T_{A_0}^i(x)) - \int_A r_{A_0}(x)d\mu_A | & < & \epsilon \\
\Rightarrow &| \frac{1}{2N} \sum_{i = 1}^{2N} r_{A_0}(T_{A_0}^i(x))  - \frac{1}{\mu(A)} | & < & \epsilon \\
\Rightarrow & \frac{1}{2N} \sum_{i = 1}^{2N} r_{A_0}(T_{A_0}^i(x)) &<& \frac{1}{\mu(A)} + \epsilon \end{array} \] for almost every $x \in A^\prime$.

Similarly, by nearly unique ergodicity of $S_{B_0}$ and as $2N > N$, we have that:
\[ \begin{array}{rrcl} {}&| \frac{1}{2N} \sum_{i = 1}^{2N} r_{B_0}(S_{B_0}^i(y)) - \int_B r_{B_0}(y)d\nu_B | & < & \epsilon \\
\Rightarrow & | \frac{1}{2N} \sum_{i = 1}^{2N} r_{B_0}(S_{B_0}^i(y))  - \frac{1}{\nu(B)} | & < & \epsilon \\
\Rightarrow & \frac{1}{2N} \sum_{i = 1}^{2N} r_{B_0}(S_{B_0}^i(y)) &>& \frac{1}{\nu(B)} - \epsilon \end{array} \] for almost every $y \in B^\prime$.

If $\epsilon < \frac{1}{2} \left( \frac{1}{\nu(B)} - \frac{1}{\mu(A)} \right)$, each pit is deeper than the pile with which it is paired, allowing $X_0$ to be mapped into $Y_0$ as the subset $\bar B$.  
\end{proof}

\begin{lemma}

$\hat \phi$ preserves the order of orbits.

\end{lemma}

\begin{proof}  

Suppose $x_1$ and $x_2$ are on the same orbit.  Let $h_1 = h(x_1)$ and $h_2 = h(x_2)$.  If $T^{-h_1}(x_1) = T^{-h_2}(x_2)$, $x_1$ and $x_2$ lie in the same pile and map to the same pit to depths $h_1$ and $h_2$ respectively, observing order.  If $T^{-h_1}(x_1) \neq T^{-h_2}(x_2)$, the points sit in different piles and map to different pits.  If $x_2 = T^m(x_1)$, $m >0$, then $x_1$ is in a pile to the left of $x_2$ and hence $\hat \phi(x_1)$ is in a pit to the left of the pit for $\hat \phi(x_2)$.  Hence, order is preserved.

\end{proof}
 
\begin{lemma}
$\hat \phi$ is nearly continuous.
\end{lemma}

\begin{proof}

Observe that one can decompose $X$ into sets $B_h$, clopen in $X_0$, for which $\hat \phi(x) = S^h \circ \phi \circ T^{-h}(x)$ for all $x \in B_h$.

\end{proof}

% Let $x_0 \in X_0$, and select $\delta_1(x_0) > 0$ small enough that if $d(x,x_0) < \delta_1$, then $h(x) = h(x_0)$.  In other words, if $x$ is close enough to $x_0$, then they enter $A'$ at the same time along the backward orbit.  Points which start close enough together map under the same composition of continuous functions: \[ \begin{array}{lcl} \hat \phi(x) &=& S^{h(x)} \circ \phi \circ T^{-h(x)} (x)\\  \hat \phi (x_0) & = & S^{h(x)} \circ \phi \circ T^{-h(x)}(x_0) \end{array}. \]  As this function is continuous, given an $\epsilon(x_0) >0$, there exists $\delta_2(x_0)$ such that $d(x,x_0) < \delta_2 $ \[ \begin{array}{rlcl} \Rightarrow & \epsilon & > & |S^{h(x)} \circ \phi \circ T^{-h(x)} (x) - S^{h(x)} \circ \phi \circ T^{h(x)} (x_0)| \\ \Rightarrow & \epsilon & > & |\hat \phi(x) - \hat \phi (x_0)| , \end{array} \] and  $\hat \phi$ is continuous at $x_0$.  Let $\delta = \min \{ \delta_1, \delta_2 \}$.  If $d(x,x_0) < \delta$, $|\hat \phi (x) - \hat \phi (x_0) | < \epsilon$.  

\begin{lemma}
$\hat \phi$ has a nearly continuous inverse.
\end{lemma}
\begin{proof}
Let $H(y) = \min \{H > 0: S^{-H}(y) \in B' \}.$  
Define \[ \hat \phi^{-1} := \left\{ \begin{array}{lcl} \phi^{-1}(y) & \mbox{ if } & y \in B' \\
											T^{H(y)} \circ \phi^{-1} \circ S^{-H(y)}(y) & \mbox{ if } & y \in (B')^C \cap Y_0 \end{array} \right. \]  Obviously, $\hat \phi ^{-1} ( \hat \phi (x)) = x$ and $\hat \phi ( \hat \phi ^{-1} (y)) = y$.  $\hat \phi ^{-1}$ is continuous as it is piece-wise continuous on a clopen decomposition of $X_0$.
\end{proof}

Finally, as $\hat \phi$ is nearly continuous, preserves the order along orbits, and has a nearly continuous inverse $\hat \phi$ is a nearly continuous conjugacy between $T$ on $X$ and $S_{\bar B}$ on $\bar B$.  Note that $\bar B$ is nearly clopen by construction.

We give an applications of Theorem\ref{NCKE} which comes into play in the following section.

\begin{lemma}\label{alpha}

Given an irrational rotation (or adding machine) $\left(Y,\tau, \nu, S \right)$, for any value $\alpha$, $0 < \alpha < 1$, there exists a nearly clopen set $B$ with $\nu(B) = \alpha$ such that $S_B$ is nearly continuously conjugate to an irrational rotation ($\alpha$ irrational), or an adding machine ($\alpha$ rational).
\end{lemma}

\begin{proof}

For $\alpha$ irrational, let $\left( X, \tau, \mu, T \right)$ be the irrational rotation of the unit interval by $\alpha$.  Let $A$ be the interval $[0, \alpha ]$, and induce on this set.  It is not hard to see that this induced transformation may be represented by the interval exchange where $[0, 1 - n \alpha ]$ maps to $[(n + 1) \alpha -1, \alpha ]$ and $[ 1 - n \alpha, \alpha ]$ maps to $[0, (n+1)\alpha - 1 ]$ where $n = \lfloor \frac{1}{\alpha} \rfloor$.  Thus, $T_A$ is an irrational rotation.  %In fact, if $\alpha = \frac{1}{a_1 + \frac{1}{a2 + \cdots}}$, then $T_A$ is the irrational rotation by $\frac{1}{a_2 + \frac{1}{a3 + \cdots}}$.

Let $(Y,\tau,\nu,S)$ be as in the hypothesis of the lemma.  As $S$ and $T$ are n.c. even Kakutani equivalent as well as nearly uniquely ergodic, there exists nearly clopen subsets $C_X \subset X$ and $C_Y \subset Y$ such that $T_{C_X}$ and $S_{C_Y}$ are n.c. conjugate.  $T^k(C_X)$ intersects with $A$ as a nearly clopen set of positive measure for some $k > 0$.  Then $\phi \circ T^k$ is an n.c. conjugacy between induced systems on $C_A = T^k(C_X) \cap A$ and $D = S^k \circ \phi (C_X \cap T^{-k}(A))$ via the map $T^k \circ \phi$ where $\phi$ is the n.c. conjugacy between $T_{C_X}$ and $S_{C_Y}$.  As $\mu(C_X) = \nu(C_y)$, by Theorem \ref{NCKE} $\mu(C_A) = \nu(D)$ and $\mu_A(C_A) > \nu(D)$.  As these systems are also nearly uniquely ergodic, there exists a nearly clopen subset $B \subset Y$ such that $T_A$ is nearly continuously conjugate to $S_B$.  Note that $\nu(B) = \alpha$. 

Next, consider the case where $\alpha$ is rational.  Write $\alpha$ as $\frac{p}{q}$.  Build an adding machine so that the first canonical tower has height $q$, for instance $\mathbb{Z} / q \mathbb{Z} \times \mathbb{Z} / q^2 \mathbb{Z} \times \mathbb{Z} / q^3 \mathbb{Z} \times \cdots$.  Let the set $A$ be the first $p$ levels of this tower.  Inducing on $A$  creates another adding machine which is isomorphic to the adding machine $\mathbb{Z} / p \mathbb{Z} \times \mathbb{Z} / pq \mathbb{Z} \times \mathbb{Z} / pq^2 \mathbb{Z} \times \cdots$.  So, $T_A$ is an adding machine, is nearly uniquely ergodic, and has measure $\mu(A) = \alpha = \frac{p}{q}$.  

Following along the lines for the case where $\alpha$ is irrational, we have the existence of a nearly clopen subset $B \subset Y$ such that $S_B$ and $T_A$ are n.c. conjugate and $\nu(B) = \alpha$. 
  Inducing on any almost clopen subset will then be n.c. even Kaktuani equivalent to both an irrational rotation or an adding machine.

\end{proof}

\section{Strongly Rank One}

The process of cutting and stacking an interval to define a transformation is a well known way of creating examples of dynamical systems.  For notation's sake, we describe the process.  By a stack of intervals, we mean a collection of intervals of equal width which are placed one above another.  To build a simple cutting and stacking construction using intervals of equal width at each stage, we need only to specify the number of columns into which to cut the stack during each stage and the number and placement of spacers (extra intervals) used.  At stage $i$, we have a stack $\Stack_i$ of subintervals of equal width from $[0,1]$, and one left over subinterval from which to cut spacers.  The transformation defined at stage $i$, $S_i$, translates each level to the level above, and the domain includes all but the topmost level and leftover interval .  Let $c(i)$ be the number of columns into which to cut the stack.  If $B_i$ represents the base of $\Stack_i$, and $n_i$ represents the height of the stack, cut $B_i$ into $c(i)$ intervals of equal width, forming $B_{i1}, \ldots, B_{ic(i)}$.  A column then consists of $ \{ B_{ij}, TB_{ij},T^2B_{ij},\ldots,T^{n_i-2}B_{ij} \}$.  Let $s^{+}(i,j)$ be the number of spacers, subintervals cut from the leftover interval, to place above the $j^{th}$ column $1 \leq j \leq c(i)$ during stage $i$, and let $s^{-}(i)$ denote the number of spacers to be placed below the first column.  To move from stage $i$ to $i + 1$, cut the stack of intervals into $c(i)$ columns of equal width, $w_{i+1}$, place $s^{+}(i,j)$ spacers of width $w_{i+1}$ above the $j^{th}$ column and $s^{-}(i)$ below the first column, and re-stack by placing the $i^{th}$ column (reading left to right) underneath column $i+1$ for $1 \leq i < c(i)$.  The space $Y$ is given by the collection of all intervals used--$[0,1]$.  The transformation, $S$ is the limit of the transformations $S_i$.  The collection of all intervals used at each stage generates the topology.

Adding machines and Chacon's map are classic examples of transformations built by such a process.  M. Roychowdhury and D. Rudolph \cite{RandR2} proved that all adding machines are evenly nearly continuously Kakutani equivalent.  A. Dykstra and D. Rudolph \cite{Dykstra} proved that irrational rotations of the circle belong to the same equivalence class.  Using these works as stepping stones, we define a class of transformations consisting of systems of a similar nature to Chacon's map, which we call strongly rank one, then prove they belong to the same equivalence class as irrational rotations.
  
\begin{defi}
A collection $\left\{ A_i: i \in \Nat \right\}$ of nearly clopen sets is called a \textbf{near basis} for the topology of $X$ if for any nearly clopen set $A \subset X$, $[A]= \bigcup_{i \in I}[A_i]$ for some $I \subset \Nat$.
\end{defi}

\begin{defi} We say that a system $(X,\tau,\mu,T)$ is \textbf{strongly rank one} if there exists a refining sequence of nearly clopen towers $\mathcal{T} _i$ such that the collection of levels is a near basis for the topology.
\end{defi}  
  
Before showing that strongly rank one transformations may be built from a cutting and stacking process, we include an essential theorem from Zhuravlev's thesis \cite{Zhuravlev}.

\begin{lemma}\label{Zhuravlev's}
Let $(X,\tau,\mu)$ and $(Y,\tau,\nu)$ be two Polish probability spaces, and let $\mathcal{A}$ and $\mathcal{B}$ denote (equivalence classes of) near bases for $X$ and $Y$, respectively.  If $\Phi: \mathcal{A} \rightarrow \mathcal{B}$ is a measure preserving isomorphism, then $\Phi$ is induced by a nearly continuous point map, i.e. there exists full subsets $X_0 \subseteq X$ and $Y_0 \subseteq Y$ and a measure preserving map $\phi: X_0 \rightarrow Y_0$ and $\phi(A) = \Phi(A)$ for each $A \in \mathcal{A}$.  
\end{lemma}

\begin{lemma}
A nearly continuous dynamical system $(X, \tau, \mu, T)$ is strongly rank one if and only if it is nearly continuously conjugate to a cutting and stacking of intervals consisting of one stack at each stage.
\end{lemma}

\begin{proof}

Suppose $(X, \tau, \mu, T)$ is strongly rank one.  Consider a refining sequence $\{ \mathcal{T} _i \}_{i = 1}^\infty $ of nearly clopen towers.  Denote by $B_i$ the base of tower $\mathcal{T}_i$, let $n_i$ be its height, and $\mathcal{P}_i^j$ for $1 \leq j \leq n_i$ be the levels of the tower as well as the name of the partition element given by the level.  At stage $i$, $X = \bigcup_{j = 1}^{n_i+1} \mathcal{P}_i^j$ and $\mathcal{P}_i^{n_i+1} = X \setminus \bigcup_{j = 1}^{n_i} \mathcal{P}_i^j $ represents the residual set and final partition element with where $\mu(\mathcal{P} _i^{n_i+1}) = \epsilon_i$.  As we move from $\mathcal{T}_i$ to $\mathcal{T}_{i + 1}$, each $\mathcal{P}_i^j = \bigcup_{k \in I_i^j} \mathcal{P}_{i+1}^k$ for some collection of indices $I_i^j \subset \mathbb{N}$ and $1 \leq j \leq n_i$.  Let $m_i = \mu(B_i)$.

To define a cutting and stacking of intervals producing a transformation which is n.c. conjugate to $(X,T)$, observe the transition from tower $\mathcal{T}_i$ to $\mathcal{T}_{i+1}$.  Let $c(i) = |I_i^1|$.  Starting from the base of $\mathcal{T}_{i+1}$, read off the names of levels according to the partition $\mathcal{P}_i$ given by $\mathcal{T}_i$.  We see complete names ($\mathcal{P}_i^1,\mathcal{P}_i^2,\ldots,\mathcal{P}_i^{n_i}$) of the tower $\mathcal{T}_i$ possibly with the name of the residual set $\mathcal{P}_i^{n_i+1}$ inserted between the complete tower names.  Let $s^{+}(i,j)$ be the number of $\mathcal{P}_i^{n_i+1}$ which occur after the $j^{th}$ appearance of the complete name of $\mathcal{T}_i$ reading up the tower $\mathcal{T}_{i+1}$, and let $s^{-}(i)$ be the number of $\mathcal{P}_i^{n_i + 1}$ which occurs before the first full tower name.

Begin with $n_1$ intervals of measure $w_1 = m_1$, and, for each $i$ and $1 \leq j \leq c(i)$, with $s^{+}(i,j) + s^{-}(i)$ intervals of width $w_i = \frac{w_1}{\Pi_{k = 1}^{i}c(k)}$.  Follow the procedure to build the cutting and stacking construction.  Let $S$ be the transformation which arises in the limit from $S_i(\mathcal{P}_i^j) = \mathcal{P}_i^{j+1}$ for $1 \leq j < n_i$ and $\mathcal{B}$ be near basis formed by the levels of the stacks.  

Let $\Phi$ be the map which arises naturally by identifying the levels of the towers for the transformation $T$ on $X$ to the sub-intervals of the stacks from the unit interval $Y$.  This is obviously a measure preserving isomorphism between the near basis consisting of nearly clopen levels from the towers for $(X,T)$ and the near basis consisting of intervals from the cutting and stacking $(Y,S)$.  By Lemma \ref{Zhuravlev's}, $\Phi$ is induced by a point map $\varphi$. $\varphi$ is nearly continuous.  Let $\epsilon > 0$ be given.  For an $x_0 \in X$, let $P$ be the element of the partition containing $x_0$ at some stage of the construction for which $\mu(P) < \epsilon$.  Let $\delta(x_0) = \frac{1}{2}d(x_0, P^C)$.  Then, for any $x$ such that $d(x,x_0) < \delta(x_0)$, $d(\varphi(x), \varphi(x_0)) < \nu(\varphi(P)) = \mu(P) < \epsilon$.  $\varphi$ preserves order.  If $x_1 = T(x_2)$, then, after some point in the construction, $x_1$ is always directly above $x_2$.  As the partitions progress, the level $P_i$ containing $x_1$ will always be directly above $P_{i-1}$ containing $x_2$.  Then, $\Phi(P_i)$ is always directly above $\Phi(P_{i-1})$ and $\varphi(x_1) = S(\varphi (x_2))$.  As the topology of the unit interval is generated by intervals whose width go to zero, $\varphi$ preserves the topologies.  Thus, $\varphi$ is a nearly continuous conjugacy.

Suppose that $(X, \tau, \mu,T)$ is nearly continuously conjugate to a cutting and stacking of the unit interval with only one stack at each stage.  As n.c. conjugacy preserves nearly clopen sets, this defines a sequence of towers for $(X,T)$ whose levels are nearly clopen, with each tower as a refinement of the previous.  These levels generate the topology.  

\end{proof}

\begin{theorem}
All strongly rank one transformations belong to the nearly continuous Kakutani equivalence class of irrational rotations.
\end{theorem}

\begin{proof}

Let $(X,\tau,\mu,T)$ be a strongly rank one transformation.  As all strongly rank one systems are n.c. conjugate to a cutting and stacking transformation, let $c(i)$ be the cuts of equal width made at stage $i$ of the cutting and stacking, and let $s^{+}(i,j)$ be the spacers placed above column $j$ for $j = 1, \ldots, c(i)$ and $s^{-}(i)$ be the number of spacers placed below column $1$.  Let $A$ be the very first interval which we cut into $c(1)$ segments.  $A$ is nearly clopen, and inducing $T_A$ on $A$ ignores the addition of all spacers, producing the adding machine which comes from the inverse limit of $\mathbb{Z} / c(1) \mathbb{Z} \times \mathbb{Z} / c(1)c(2) \mathbb{Z} \times \mathbb{Z} / c(1)c(2)c(3) \mathbb{Z} \times \cdots \times \mathbb{Z} / \Pi_{n=1}^ic(n) \mathbb{Z} \times \cdots$.  Suppose $\mu(A) = \alpha$.

By Lemma \ref{alpha}, given any irrational rotation (or adding machine) $(Y,\tau,\nu,S)$, there exists a nearly clopen set $B \subset Y$ with $\mu(B) = \alpha$ such that $S_B$ is n.c. conjugate to either an irrational rotation or an adding machine.  Then $T_A$ is n.c. even Kakutani equivalent to $S_B$ and the result follows.
\end{proof}

\section{Conclusion}

We extended the equivalence class to include all strongly rank one transformations, such as Chacon's map.  It is interesting to note that all such transformations enjoy a canonical cutting and stacking representation relied upon in \cite{RandR2} and \cite{Dykstra}.  In order to obtain the nearly clopen sets upon which they define the conjugacy between induced transformations, they must first establish a n.c. orbit equivalence.  Here we did not need to establish an orbit equivalence as we started with a conjugacy between the induced transformations on the nearly clopen sets, leading us to shift the definition of a nearly continuous even Kakutani equivalence to just the existence of a nearly continuous conjugacy between induced maps as in the measure-theoretical analogue.  Again, we query about systems which lack a canonical cutting and stacking, a prime example being minimal isomorphisms of compact metric spaces.  All irrational rotations of a circle are known to be measure theoretically rank one.   Are all irrational rotations strongly rank one?

\end{document}